\documentclass[11pt, a4paper, dvipsnames,reqno]{amsart}

\pdfoutput=1 % for arxiv.org

\usepackage{geometry}
\usepackage[english]{babel}
\usepackage[utf8]{inputenc}
\usepackage[T1]{fontenc}
\usepackage{amsmath, amsthm, amssymb, amsfonts}
\usepackage{tikz}
\usepackage{caption}
\usepackage{subcaption}
\usepackage{tkz-euclide}
\usepackage{mathtools, wasysym, pifont}
\usepackage{graphicx}
\usepackage{adjustbox}
\usepackage[inline]{enumitem}
\usepackage{tasks}
\usepackage{fancyhdr}
\usepackage[skip=0.2\baselineskip, indent=1em]{parskip}
\usepackage{tkz-euclide}
\usepackage{mdframed}
\usepackage[pdfencoding=auto, colorlinks=true, linkcolor=blue, citecolor=blue]{hyperref}
\usepackage{tcolorbox}
\usepackage[textsize=scriptsize]{todonotes}
\usetikzlibrary{matrix,decorations.pathreplacing,calc}
\usepackage{xcolor}

\usetikzlibrary{arrows.meta}
\usetikzlibrary{arrows,positioning}

\usepackage{tikz}
\usetikzlibrary{matrix,decorations.pathreplacing,calc}
\usetikzlibrary{decorations.pathmorphing, patterns,shapes, snakes}
\usetikzlibrary{patterns}

\pgfkeys{tikz/mymatrixenv/.style={decoration=brace,every left
    delimiter/.style={xshift=3pt},every right
    delimiter/.style={xshift=-3pt}}}

\pgfkeys{tikz/mymatrix/.style={matrix of math nodes,left
    delimiter=(,right delimiter={)},inner sep=6pt,column sep=1em,row
    sep=0.5em,nodes={inner sep=0pt}}}

\pgfkeys{tikz/mymatrixbrace/.style={decorate,thick}}
\newcommand\mymatrixbraceoffseth{0.9em}
\newcommand\mymatrixbraceoffsetv{0.5em}

\newcommand*\mymatrixbraceright[4][m]{
    \draw[mymatrixbrace, black!70] ($(#1.north east)!(#1-#2-1.north east)!(#1.south east)+(\mymatrixbraceoffseth,-2pt)$)
        -- node[right=2pt] {#4} 
        ($(#1.north east)!(#1-#3-1.south east)!(#1.south east)+(\mymatrixbraceoffseth,-2pt)$);
}
\newcommand*\mymatrixbracetop[4][m]{
    \draw[mymatrixbrace, black!70] ($(#1.north west)!(#1-1-#2.north west)!(#1.north east)+(-9pt,\mymatrixbraceoffsetv)$)
        -- node[above=2pt] {#4} 
        ($(#1.north west)!(#1-1-#3.north east)!(#1.north east)+(7pt,\mymatrixbraceoffsetv)$);
}
\newcommand*\mymatrixbracetopalt[4][m]{
    \draw[mymatrixbrace, black!70] ($(#1.north west)!(#1-1-#2.north west)!(#1.north east)+(-16pt,\mymatrixbraceoffsetv)$)
        -- node[above=2pt] {#4} 
        ($(#1.north west)!(#1-1-#3.north east)!(#1.north east)+(5pt,\mymatrixbraceoffsetv)$);
}

\newtheoremstyle{Def}
{0.3 cm} %space above
{0.3 cm} %space below
{\normalfont} %bodyfont
{} %indentamt
{\itshape} %headfont
{.} %headpunct
{0.25cm} %headspace
{}%headspec

\newtheoremstyle{lemma}
{0.3 cm} %space above
{0.3 cm} %space below
{\normalfont} %bodyfont
{} %indentamt
{\bfseries} %headfont
{.} %headpunct
{0.25cm} %headspace
{}%headspec

\theoremstyle{lemma}

\makeatletter
\newtheorem*{rep@theorem}{\rep@title}
\newcommand{\newreptheorem}[2]{%
	\newenvironment{rep#1}[1]{%
		\def\rep@title{#2 \ref{##1}}%
		\begin{rep@theorem}}%
		{\end{rep@theorem}}}
\makeatother

\newtheorem{lemma}{Lemma}[section]
\newtheorem{fact}[lemma]{Fact}
\newtheorem{theorem}[lemma]{Theorem}
\newtheorem{proposition}[lemma]{Proposition}
\newtheorem{corollary}[lemma]{Corollary}
\newtheorem{definition}[lemma]{Definition}
\newreptheorem{theorem}{Theorem}
\newreptheorem{lemma}{Lemma}
\newreptheorem{proposition}{Proposition}

\theoremstyle{Def}
\newtheorem{notation}[lemma]{Notation}
\newtheorem{observation}[lemma]{Observation}
\newtheorem{example}[lemma]{Example}

\newtheorem{remark}[lemma]{Remark}
\newtheorem{convention}[lemma]{Convention}

\DeclareMathOperator{\Ass}{Ass}

\newcommand{\N}{\mathbb{N}}
\newcommand{\Z}{\mathbb{Z}}

\DeclareMathOperator{\Span}{span}

\DeclareMathOperator{\cone}{cone}
\DeclareMathOperator{\conv}{conv}

\newcommand{\fm}{\mathfrak{m}}
\newcommand{\fp}{\mathfrak{p}}

\newcommand{\cpi}[1]{\mathsf{B}_{\supseteq}^{#1}}
\newcommand{\persind}[1]{\mathsf{B}_{\subseteq}^{#1}}
\newcommand{\stabind}[1]{\mathsf{B}_{=}^{#1}}

\newcommand{\dred}{d_{\operatorname{red}}}

\newcommand{\longmid}{\,\,\middle\vert\,\,}
\renewcommand{\vec}[1]{\boldsymbol{#1}}
\newcommand{\veca}{\vec{a}}

\newcommand{\sat}[1][I]{\operatorname{sat}(#1)}
\newcommand{\indi}[1]{#1_{I^n:\fm}}
\newcommand{\indii}[1]{#1_{\sat[I^n]}}
\newcommand{\transpose}{\mathsf{T}}
\newcommand{\Sb}{\mathfrak{S}_{\vec{b}}}
\newcommand{\SO}{\mathfrak{S}_{\vec{0}}}

\author{Clemens Heuberger}
\address{Institut für Mathematik\\Alpen-Adria-Universität Klagenfurt\\
  Universitätsstraße 65-67\\9020 Klagenfurt am Wörthersee\\Austria}
\email{\href{mailto:clemens.heuberger@aau.at}{clemens.heuberger@aau.at}}

\author{Jutta Rath}
\address{Institut für Mathematik\\Alpen-Adria-Universität Klagenfurt\\
  Universitätsstraße 65-67\\9020 Klagenfurt am Wörthersee\\Austria}
\email{\href{mailto:jutta.rath@aau.at}{jutta.rath@aau.at}}

\author{Roswitha Rissner}
\address{Institut für Mathematik\\Alpen-Adria-Universität Klagenfurt\\
  Universitätsstraße 65-67\\9020 Klagenfurt am Wörthersee\\Austria}
\email{\href{mailto:roswitha.rissner@aau.at}{roswitha.rissner@aau.at}}
\thanks{This research was funded in part by the Austrian
  Science Fund (FWF) [10.55776/DOC78]. For open access purposes, the
  authors have applied a CC BY public copyright license to any
  author-accepted manuscript version arising from this submission.}

\title[Bounding the copersistence index]{Stabilization of associated prime ideals of monomial ideals -- Bounding the copersistence index}


\begin{document}

\begin{abstract}
  The sequence $(\Ass(R/I^n))_{n\in\N}$ of associated primes of powers
  of a monomial ideal $I$ in a polynomial ring $R$ eventually stabilizes
  by a known result by Markus Brodmann. Lê Tuân Hoa
  gives an upper bound for the index where the stabilization
  occurs. This bound depends on the generators of the ideal and is
  obtained by separately bounding the powers of $I$ after which said
  sequence is non-decreasing and non-increasing, respectively.  In
  this paper, we focus on the latter and call the smallest such number
  the copersistence index.  We take up the proof idea of Lê Tuân Hoa,
  who exploits a certain system of inequalities
  whose solution sets store information about the associated primes of
  powers of~$I$. However, these proofs are entangled with a specific
  choice for the system of inequalities.  In contrast to that, we present
  a generic ansatz to obtain an upper bound for the copersistence
  index that is uncoupled from this choice of the system. We establish
  properties for a system of inequalities to be eligible for this
  approach to work.  We construct two suitable inequality systems to
  demonstrate how this ansatz yields upper bounds for the
  copersistence index and compare them with Hoa's.
  One of the two systems leads to an improvement of the bound by an exponential factor.
\end{abstract}

\maketitle

\section{Introduction and Overview}
The set $\Ass(R/I)$ of associated primes gives insight into the structure of an ideal $I$ in a Noetherian ring $R$. For example, the associated primes of an integer ideal $n\Z$ are generated by the prime factors of $n$, and the associated primes of an algebraic variety correspond to its irreducible components (cf.~\cite[Section~3.8]{Eisenbud:2004:commutativealgebra}). For edge ideals of simple graphs, the associated primes correspond to minimal vertex covers (cf.~\cite[Lemma~2.13]{Carlini-Ha-Harbourne-Tuyl:powers-of-ideals}). The cover ideal is the Alexander dual of the edge ideal, and its associated primes correspond to the edges of the graph (cf.~\cite[Lemma~2.12]{Carlini-Ha-Harbourne-Tuyl:powers-of-ideals}). For finite simple hypergraphs, there is an intrinsic relation between the associated primes of powers of the cover ideal and coloring properties of the underlying graph~\cite{Francisco-Ha-VanTuyl:2011:assPrimes-perfectgraphs}.

The changes of the sequence $(\Ass(R/I^n))_{n\in\N}$ in $n$ are being studied over the last few decades for various classes of ideals. 
For some classes of ideals, the sequence of sets is known to be growing, e.g.,~for edge ideals of simple, undirected graphs~\cite[Theorem~2.15]{MartinezBernal-Morey-Villarreal:2012:assPrimes-edge}, for cover ideals of perfect graphs~\cite[Corollary~5.11]{Francisco-Ha-VanTuyl:2011:assPrimes-perfectgraphs}, or for ideals whose powers are all integrally closed~\cite[Theorem~2.4]{Ratliff:1984:asymptotic-prime-divs}. However, even for square-free monomial ideals, this so-called persistence property does not hold in general~\cite[Theorem~11]{Kaiser-Stehlik-Skrekovski:2014:persistence}. Sarah Weinstein and Irena Swanson~\cite[Theorem~3.9]{Weinstein-Swanson:2020:decay} construct families of monomial ideals whose sets of associated prime ideals decrease in $n$.
There are examples known where $(\Ass(R/I^n))_{n\in\N}$ is not even monotone (cf.~\cite[p.~80]{MCADAM197971}).
In general, even for monomial ideals, only little is known about the changes of $\Ass(R/I^n)$ in $n$.

Despite that, the asymptotic behavior of this sequence is well understood: In 1979, Markus Brodmann~\cite{Brodmann:1979:asympstab} proved that the sequence of associated primes of powers of an ideal stabilizes, meaning that for large $n\in\N$
\begin{equation*}
\Ass(R/I^n)=\Ass(R/I^{n+1})
\end{equation*}
holds.
The smallest integer $\stabind{I}$ such that $\Ass(R/I^n)=\Ass(R/I^{\stabind{I}})$ for all $n\ge \stabind{I}$ is called the \textbf{stability index}, and the set $\Ass(R/I^{\stabind{I}})$ is called the \textbf{stable set} of $I$.

One question arising is whether there exists a universal bound for the stability index.
Lê Tuân Hoa~\cite{Hoa:2006:stab-assprimes-monom} gives an upper bound for the stability index of monomial ideals 
depending on the number of variables, the number of generators of the ideal, and their maximal total degree.
In the same publication, examples of families of ideals demonstrate the dependence of any such bound on the number of variables and the degrees of the generators. 

In general, this bound is very large; for example, for the ideal $I=(XY,YZ)$ in $K[X,Y,Z]$, it is greater than $8\cdot 10^7$, whereas the actual stability index is $1$, cf.~\cite[Example~2.17]{Carlini-Ha-Harbourne-Tuyl:powers-of-ideals}.
For some classes of monomial ideals, tighter bounds are known.
Jürgen~Herzog even conjectured that, if $r$ denotes the number of variables, then for square-free monomial ideals, this bound can be reduced to $r-1$, cf.~\cite[Section~2.3]{Carlini-Ha-Harbourne-Tuyl:powers-of-ideals}.
A lot of research in that area focuses on edge and cover ideals of graphs; see for example~\cite{Chen-Janet-Morey:stable-set-graph, Francisco-Ha-VanTuyl:2011:assPrimes-perfectgraphs, Lam-2019, MartinezBernal-Morey-Villarreal:2012:assPrimes-edge, simis-ideal-theory-graphs, Terai-Trung:2014:associated-primes-depth}. Also other classes of ideals have been studied over the last decades,  cf.~\cite{HERZOG2015530, Herzog-Rauf-Vladoiu:2013:polynomatroidal, Khashyarmanesh-Nasernejad-2014:stab-set, Trung:2009:stab-assoc-primes}.

Lê Tuân Hoa~\cite{Hoa:2006:stab-assprimes-monom} establishes said upper
bound for the stability index of monomial ideals by separately bounding the
power after which the sequence $(\Ass(R/I^n))_{n\in\N}$ is
non-increasing and non-decreasing, respectively. 
We call $\persind{I}$ the \textbf{persistence index} of $I$ if
$\persind{I}$ is the smallest integer such that
\begin{equation*}
\Ass(R/I^n)\subseteq\Ass(R/I^{n+1}) \text{ for all } n\ge \persind{I}.
\end{equation*}
An ideal where $\persind{I}=1$ is said to fulfill the \textbf{persistence property}.
For instance, edge ideals of finite simple graphs~\cite{MartinezBernal-Morey-Villarreal:2012:assPrimes-edge} and cover ideals of perfect graphs~\cite{Francisco-Ha-VanTuyl:2011:assPrimes-perfectgraphs} satisfy the persistence property. 
We call $\cpi{I}$ the \textbf{copersistence index} of $I$ if
$\cpi{I}$ is the smallest integer such that
\begin{equation*}
\Ass(R/I^n)\supseteq\Ass(R/I^{n+1}) \text{ for all } n\ge \cpi{I}.
\end{equation*}
With these definitions, the stability index is the maximum of the persistence and the copersistence index.
This paper focuses on improving upper bounds for the copersistence index.

\subsection{Results}
We develop a general method to determine upper bounds for the copersistence index $\cpi{I}$ of a monomial ideal $I$ in the polynomial ring $K[X_1,\dots, X_r]$, where $K$ is an arbitrary field. This method is based on describing membership in monomial ideals by suitable systems of linear inequalities.
We set up such a system which leads to a reduction of the existing bound~\cite{Hoa:2006:stab-assprimes-monom} for the copersistence index by an exponential factor
\[\frac1{\sqrt{2r}}\Bigl(\frac{d\sqrt{2}}{\sqrt{d^2+1}}\Bigr)^{rs}\]
where $s$ is the number of generators of the ideal and $d$ is the maximal total degree of these generators, see Proposition~\ref{thm:compare-bounds}; note that $\frac{d\sqrt{2}}{\sqrt{d^2+1}}\ge \sqrt{8/5}>1$.
 
All associated primes of monomial ideals are of the form $\fp(M)\coloneqq(X_i\mid i\in M)$ for some subset $M\subseteq[r]$, where $[r]$ denotes the set $\{1,\dots, r\}$ (Fact~\ref{fact:associated-primes-of-monomial-ideals}).
Since the sequence $(\Ass(R/I^n))_{n\in \N}$ eventually stabilizes, ``copersistence behavior'' occurs for each prime ideal $\fp(M)$ at some point.
In other words, for every $M\subseteq[r]$ there exists an integer $N$ such that, if $\fp(M)$ is not associated to $I^n$ for some $n\ge N$, then $\fp(M)$ is also not associated to any higher power.
We denote the smallest such integer by $\cpi{I}(M)$ (Definition~\ref{definition:B1(m)}) and determine upper bounds.

For that purpose, we use systems of linear inequalities $A\vec{x}\le \vec{b}$ that describe to which powers of $I$ the ideal $\fp(M)$ is associated.
From the solutions of the system, we derive, what we call, the $n$\textbf{-th solution spaces} $H_{\vec{0},n}$ and $H_{\vec{b},n}$ (Definition~\ref{def:H,U}), which are groups that store the information whether $\fp(M)$ is associated to~$I^n$. 

In Theorem~\ref{theorem:bounds-sigma(m)}, we show how such systems of linear inequalities can be used to give a bound for~$\cpi{I}(M)$ depending on $\Delta(A\mid \vec{b})$, the maximal absolute value of the subdeterminants of the augmented matrix $(A\mid \vec{b})$.

\begin{reptheorem}{theorem:bounds-sigma(m)}
	Let $I$ be a monomial ideal in $K[X_1,\dots, X_r]$ and $M\subseteq[r]$. Further, let
$A\in\Z^{m\times\nu}$ and $\vec{b}\in\Z_{\ge0}^m$ such that for every $n\in\N$ the associated $n$-th solution spaces $H_{\vec{0},n}$ and $H_{\vec{b},n}$ (Definition~\ref{def:H,U})
fulfill

\begin{enumerate}
\item\label{condition:m-ass-equivalence} $\fp(M)\in\Ass(R/I^n)$ if and
  only if $H_{\vec{b},n}/H_{\vec{0},n}\neq 0$, and
\item\label{condition:HnHm} for all $n_1$, $n_{2}\in\N$ we have
  $H_{\vec{0},n_1}H_{\vec{0},n_2}=H_{\vec{0},n_1+n_2}$.
\end{enumerate}
Then $\cpi{I}(M)\le \Delta(A\mid \vec{b})(\nu+1)$.

%%% Local Variables:
%%% mode: latex
%%% TeX-master: "bounding-the-copersistence-index.tex"
%%% End:
\end{reptheorem}

It is well-known that associated primes of ideals behave well with respect to localization (Fact~\ref{fact:localization}). This allows us to focus on maximal ideals (Observation~\ref{observation:p(M)}). 

The obtained bound for $\cpi{I}([r])$ from Theorem~\ref{theorem:bounds-sigma(m)} clearly depends on the system $A\vec{x}\le\vec{b}$. 
For certain systems it is, however, possible to further estimate $\cpi{I}([r])$ by a new bound $\sigma(d,s,r)$ that only depends on the number of variables $r$, the number $s$ of generators of $I$, and their maximal total degree $d$. Indeed, as the next proposition states, whenever such a function $\sigma$ exists that is non-decreasing in all variables, this yields a bound not only for $\cpi{I}([r])$, but for the copersistence index $\cpi{I}$. 

\begin{repproposition}{proposition:sigma}
	Let $\sigma\colon\N^3\to\N$ be a map that is non-decreasing in all three
variables such that for all $d$, $s$, $r\in\N$ the inequality $\cpi{I}([r])\le\sigma(d,s,r)$ holds, whenever $I$ is a monomial ideal in $r$ variables, $s$ generators, and whose minimal generators have total degree at most $d$.	
Then $\sigma(d,s,r)$ is an upper bound for the copersistence index
$\cpi{I}$ of every such ideal $I$.

%%% Local Variables:
%%% mode: latex
%%% TeX-master: "bounding-the-copersistence-index.tex"
%%% End:
\end{repproposition}

From a methodical point of view, Theorem~\ref{theorem:bounds-sigma(m)} and Proposition~\ref{proposition:sigma} are the core results of this paper. Their combination provides a generic ansatz to obtain an upper bound for the copersistence index via inequality systems. 

We emphasize that this methodological approach is based on the proof idea pursued by Lê Tuân Hoa~\cite{Hoa:2006:stab-assprimes-monom}. There, however, the proofs are entangled with a specific choice of a system of inequalities. The benefit of Theorem~\ref{theorem:bounds-sigma(m)} combined with Proposition~\ref{proposition:sigma} is that this ansatz is uncoupled from the choice of the system in the sense that it only requires specified system properties. 
It provides the possibility to determine different upper bounds for the copersistence index by choosing different suitable systems of inequalities.  This is convenient, as the structure of the system influences the estimate on $\Delta(A\mid \vec{b})$. The ansatz also lays the foundation for future work to derive tailored upper bounds for the copersistence index of a specific ideal or family of ideals.
Depending on the chosen inequality system, this approach potentially provides sharper estimates on $\Delta(A\mid \vec{b})$.

We apply this ansatz to obtain an upper bound for $\cpi{I}$ for a general monomial ideal $I$.

\begin{reptheorem}{theorem:bound-B1}
	Let $I$ be a monomial ideal in the
ring $K[X_1,\dots, X_r]$ with $s$ generators, reduced maximal degree $\dred$
(Definition~\ref{def:d-red}), and
\begin{equation*}
  \sigma_2(d,s,r)\coloneqq(\sqrt{d^2+1})^{r s}\left(\sqrt{r}\right)^{r+2}(r s+r+2). 
\end{equation*}
	
Then $\cpi{I} \le \sigma_2(\dred,s,\min\{r,s\})$ holds.

%%% Local Variables:
%%% mode: latex
%%% TeX-master: "bounding-the-copersistence-index.tex"
%%% End:
\end{reptheorem}

Finally, we compare the bound from Theorem~\ref{theorem:bound-B1} with the original result of Lê Tuân Hoa~\cite{Hoa:2006:stab-assprimes-monom} who proves that 
\begin{equation*}
	\sigma_1(d,s,r) \coloneqq d(rs+s+d)(\sqrt{r})^{r+1}(\sqrt{2}d)^{(r+1)(s-1)}
\end{equation*}
is an upper bound for $\cpi{I}$.

\begin{repproposition}{thm:compare-bounds}
	Let $2\le r \le s$ and $d\ge 2$. Then
\begin{equation*}
	\sigma_2(d,s,r)<\frac{q(d)^{rs}}{\sqrt{2r}}\cdot\sigma_{2}(d,s,r)\le\sigma_{1}(d,s,r)
\end{equation*}
holds, where $q(d)\coloneqq\frac{d\sqrt{2}}{\sqrt{d^2+1}}>1$.

\end{repproposition}

We conclude the introduction with a short overview on the structure of this paper. In Section~\ref{section:preliminaries} we introduce preliminary material on associated primes of monomial ideals, and give equivalent conditions for the maximal monomial ideal to be associated to $I^n$ (Remark~\ref{remark:characterization-m-associated}).
As outlined above, we translate these conditions into the language of systems of linear inequalities.
We want to distinguish between general results about solution spaces of such systems and their applications to our primary question on the copersistence index.
In Section~\ref{sec:bounds} we develop the methodological results on the former, whereas in Section~\ref{section:upper-bound} we return to our original problem, and translate the abstract results into concrete statements about upper bounds for the copersistence index.
Finally, in Section~\ref{section:compare-bounds}, we compare the bound obtained in Section~\ref{section:upper-bound} with known results.

\section{Preliminaries: Monomial ideals and their associated primes}\label{section:preliminaries}
We introduce some terminology and facts about monomial ideals and associated primes of monomial ideals.
For a thorough introduction we refer to Chapter~1 in Jürgen Herzog's and Takayuki Hibi's textbook~\cite{Herzog-Hibi:2011:monideals} on monomial ideals.
\subsection{Notation and basic facts about monomial ideals}
\begin{convention}
	Throughout this paper, if not explicitly stated otherwise, let $K$ be a field and $R\coloneqq K[X_1,\dots, X_r]$ be the polynomial ring over $K$ in $r$ variables.
\end{convention}
\begin{convention}
	For the set of natural numbers, we use the notation $\N\coloneqq\{1,2,3,\dots\}$, and furthermore $\N_0\coloneqq\{0,1,2,\dots\}$.
	For a natural number $n\in\N$ we sometimes write $[n]\coloneqq\{1,2,\dots,n\}$.
\end{convention}
\begin{notation}\label{notation:a[i]}
	For $\vec{b}=(b_1,\dots, b_r)^{\transpose}\in\N_0^r$, we use the notation
	\begin{equation*}
		X^{\vec{b}}\coloneqq X_1^{b_1}\cdots X_r^{b_r}.
	\end{equation*}
	If an ideal $I$ is generated by monomials $X^{\veca_1}$, \dots, $X^{\veca_s}$ for $\veca_1$, \dots, $\veca_s\in\N_0^r$, we write
	\begin{equation*}
		I = (X^{\veca_1},\dots, X^{\veca_s}).
	\end{equation*}
\end{notation}

Throughout this paper, we focus on \textbf{monomial ideals}, that is, ideals generated by monomials.

The following fact states that a monomial ideal can be fully described by the monomials contained in that ideal, and furthermore, we see that monomial ideals behave nicely under algebraic operations, as listed below.
\begin{fact}[{cf.~\cite[Theorem~1.1.2 and Chapter~1.2]{Herzog-Hibi:2011:monideals}}]\label{fact:monomial-ideals}
	Let $I$, $J\subseteq R$ be monomial ideals.
	
	Then the following properties hold:
	\begin{enumerate}
		\item The set of monomials belonging to $I$ forms a $K$-basis of $I$.
		\item The intersection $I\cap J$ is a monomial ideal.
		\item The product $I\cdot J$ is a monomial ideal. In particular, $I^n$ with $n\in\N$ is a monomial ideal.
		\item\label{property:colon-ideal-is-monomial} The colon ideal $I:J=\{w\in R\mid wJ\subseteq I\}$ is a monomial ideal.
		\item The radical $\sqrt{I}=\{w\in R\mid \exists n\in\N: w^n\in I\}$ is a monomial ideal.
	\end{enumerate}
\end{fact}
We will often implicitly use these properties of monomial ideals without always referring to Fact~\ref{fact:monomial-ideals}.
Especially Property~\eqref{property:colon-ideal-is-monomial} will come in handy many times. In that context, we introduce some more notation for colon ideals.
\begin{notation}
	Let $I$, $J\subseteq R$ be two monomial ideals, and  let $X^{\veca}$ be a monomial in $R$. We write
	\begin{enumerate}
		\item $I:X^{\veca}\coloneqq I:(X^{\veca})=\{w\in R\mid wX^{\veca}\in I\}$, and
		\item $I:J^{\infty}\coloneqq\bigcup_{k\in\N_0}(I:J^k)$.
	\end{enumerate}
\end{notation}

\begin{remark}\label{remark:reduction-of-variables}
	Let us choose a variable $X_j$, $j\in[r]$ and consider the ideal $\widetilde{I}\coloneqq I:X_j^{\infty}$.
	Then $\widetilde{I}$ is the ideal generated by the same monomials as $I$, but every $X_j$ is replaced by $1$.
	That holds because a monomial  $u$ is an element of $\widetilde{I}$ if and only if there exists a generator $g$ of $I$ such that $g\mid uX_j^n$ for some $n\in\N_0$. This implies that $u$ is divisible by the monomial that is obtained by replacing $X_j$ by $1$ in $g$. The other inclusion is clear.
	Hence, the ideal $\widetilde{I}$ is obtained by eliminating all the powers of $X_j$ from the generators in $I$.
\end{remark}

\subsection{Observations on associated primes of monomial ideals}
\begin{definition}
	For any ideal $I\subseteq R$, a prime ideal $\mathfrak{p}$ is an \textbf{associated prime} of $I$ if there exists an element $w\in R$ such that $\fp = I:w$.
	The element $w$ is called a \textbf{witness} of $\fp$ with respect to $I$.
	The set of associated primes of $I$ is denoted by $\Ass(R/I)$.
\end{definition}

For monomial ideals, the set of associated primes can be described as follows:
\begin{fact}[cf.~{\cite[Corollary~1.3.10]{Herzog-Hibi:2011:monideals}}]\label{fact:zeugen}
	For a monomial ideal $I$ we have
	\begin{equation*}
	\Ass(R/I)=\{\fp\text{ prime ideal} \mid \text{ there exists }X^{\veca}\in R\text{ such that }\fp=I:X^{\veca}\},
	\end{equation*}
	that is, monomial witnesses always exist.	
\end{fact}

An immediate corollary of Fact~\ref{fact:zeugen} in combination with Fact~\ref{fact:monomial-ideals} is that all associated primes of monomial ideals are again monomial ideals.
More precisely, we see below (Fact~\ref{fact:associated-primes-of-monomial-ideals}) that they are ideals generated by subsets of the variables.

\begin{notation}\label{notation:I(M),R(M),p(M)}
	Let $M$ be a subset of $[r]$. We use the notation
	\begin{equation*}
		\fp(M)\coloneqq(X_i\mid i\in M)\subseteq R
	\end{equation*}
	for the prime ideal generated by the variables $X_i$, $i\in M$.
	If $M=[r]$ we will denote by
	\begin{equation*}
		\fm_R\coloneqq (X_1,\dots, X_r)=\fp([r])
	\end{equation*} the unique maximal monomial ideal in $R$.
	We simply write $\fm$ if the ring is clear from the context.
\end{notation}
\begin{fact}[cf.~{\cite[Section~1.3]{Herzog-Hibi:2011:monideals}}]\label{fact:associated-primes-of-monomial-ideals}
	Let $I\subseteq K[X_1,\dots,X_r]$ be a monomial ideal and $S=\{i\in [r]\mid X_i\text{ divides a minimal generator of } I\}$. Then all associated primes of $I$ are prime monomial ideals with generators in $S$, that is,
	\begin{equation*}
		\Ass(R/I)\subseteq\{\fp(M)\mid M\subseteq S\}.
	\end{equation*}
\end{fact}

\begin{notation}\label{notation:localization}
	Let $M\subseteq[r]$. We will denote by
	\begin{equation*}
	R_M\coloneqq  (K[X_i\mid i\notin M]\setminus\{0\})^{-1}R
	\end{equation*}
	the localization of $R$ at $K[X_i\mid i\notin M]\setminus\{0\}$.
	If $I$ is an ideal in $R$, then we write
	\begin{equation*}
	I_M\coloneqq IR_M.
	\end{equation*}
\end{notation}

\begin{remark}\label{remark:localization}
	In contrast to the localization of $R$ at $\fp(M)$, $R_M$ is again a polynomial ring over a (larger) field and has $\lvert M\rvert$ variables, that is, 
	$R_M = K'[X_i\mid i\in M]$ where $K'= K(X_i\mid i\notin M)$.
\end{remark}

\begin{fact}[{cf.~\cite[Theorem~3.1]{Eisenbud:2004:commutativealgebra}}]\label{fact:localization}
	Associated primes of ideals behave well with respect to localization, that is,
	\begin{equation*}
	\Ass(R_M/I_M)=\{\fp R_M\mid\fp\in\Ass(R/I)\text{ and }X_i\notin \fp \text{ for all }i\in [r]\setminus M\}.
	\end{equation*}
\end{fact}

\begin{observation}\label{observation:p(M)}
	Let $M\subseteq[r]$. By Fact~\ref{fact:localization} we have
	\begin{equation*}
	\fp (M)\in\Ass(R/I)\quad\Longleftrightarrow\quad \fm_{R_M}=\fp (M)R_M\in\Ass(R_M/I_M).
	\end{equation*}
	This equivalence allows us to focus on the maximal ideal $\fm_R$ only.
	For non-maximal prime ideals $\fp (M)$ we localize to $R_M$ where $\fp (M)R_M$ is maximal.
	To sum up, the following holds:
	\begin{align*}
		\Ass(R/I)
		=\bigcup_{M\subseteq[r]}\big\{\fp (M)\mid\fm_{R_M}\in\Ass(R_M/I_M)\big\}.
	\end{align*}
\end{observation}

\subsection{The copersistence index}
As outlined in the introduction, when considering powers of an ideal, the associated primes of that ideal can change and eventually stabilize.
Recall that we denote the stabilization index of an ideal $I$ by $\stabind{I}$.

Lê Tuân Hoa~\cite{Hoa:2006:stab-assprimes-monom} established an upper
bound for $\stabind{I}$ of monomial ideals by separately bounding the
power after which the sequence $(\Ass(R/I^n))_{n\in\N}$ is
non-increasing and non-decreasing, respectively.
We focus on giving upper bounds for the former.

\begin{definition}\label{def:cpi-persindex}
      For an ideal $I$, we call $\cpi{I}$ the \textbf{copersistence index} of $I$ if
      $\cpi{I}$ is the smallest integer such that
      \begin{equation*}
	\Ass(R/I^n)\supseteq\Ass(R/I^{n+1}) \text{ for all } n\ge \cpi{I}.
      \end{equation*}
\end{definition}

\begin{remark}\label{remark:bound-Hoa}
	Lê Tuân Hoa~\cite{Hoa:2006:stab-assprimes-monom} proved that for a monomial ideal $I$, we have
	\begin{equation*}
		  \cpi{I} \le d(rs+s+d)(\sqrt{r})^{r+1}(\sqrt{2}d)^{(r+1)(s-1)},
    \end{equation*}
     where $r$ is the number of variables, $s$ is the number of generators of $I$ and $d$ is the maximal total degree of a minimal generator of $I$ (see Subsection~\ref{subsec:Un=JnI^n-1}).
     Apart from that, little is known about $\cpi{I}$ in general.
\end{remark}

\begin{example}\label{example:stabind<=2}
	Let $I=(X_1^4, X_1^3X_2, X_1^2X_2^2X_3,X_1X_2^3,X_2^4)\le K[X_1,X_2,X_3]$.
	For a monomial ideal, there is a straightforward way to decide whether it is primary: The ideal is primary if and only if for every variable that divides any of the minimal generators, some power of this variable is an element of the ideal.
	Here, $I$ is not primary, since $X_3\mid X_1^2X_2^2X_3$ but no power of $X_3$ is a generator of $I$.
	However, using the same technique, we see that
	\begin{equation*}
		I^2=(X_1^8, X_1^7X_2, X_1^6X_2^2, X_1^5X_2^3, X_1^4X_2^4, X_1^3X_2^5, X_1^2X_2^6, X_1X_2^7, X_2^8)
	\end{equation*}
	is primary.
	Furthermore, $I^3=(X_1^a X_2^b\mid a+b=12, a\ge 0, b\ge 0)$ is primary as well and with that, every power $I^n$ with $n\ge 2$ is primary.
	Hence $\cpi{I}\le\stabind{I}\le2$, whereas with $r=3$, $s=5$ and $d=5$ the bound for $\cpi{I}$ given in Remark~\ref{remark:bound-Hoa} is $>4\cdot10^{16}$.
	In Section~\ref{section:upper-bound} we develop an upper bound which reduces this by a factor of $10^3$.
\end{example}

\subsection{Notes on the parameters \texorpdfstring{$d$, $s$, and $r$}{d, s, and r} in the bound}\label{section:conditions-on-generators}
First, we show that common divisors of the generators of
$I$ do not play a role for the associated primes of $I^n$. 

\begin{lemma}\label{lemma:verschieben}
	Let $I$ be a monomial ideal in $r>1$ variables and let $\vec{t}\in\N_0^r$ such that $X^{\vec{t}}$ divides all of the generators of $I$.  Then for all $n\in\N$,
	$\fm\in\Ass(R/I^n)$ if and only if $\fm\in\Ass(R/(I:X^{\vec{t}})^n)$.
\end{lemma}
\begin{proof}
	It suffices to show the assertion for $n=1$ since, using that $I^n\cap (X^{n\vec{t}})=I^n$ and the fact that for any $w\in R\setminus\{0\}$ we have $I:w = \frac{1}{w}(I\cap (w))$, the following equation follows:
	\begin{equation*}
	I^n:X^{n{\vec{t}}}=\frac{1}{X^{n{\vec{t}}}}(I^n\cap (X^{n{\vec{t}}}))=\frac{1}{X^{n{\vec{t}}}}I^n=\left(\frac{1}{X^{\vec{t}}}I\right)^n=\left(\frac{1}{X^{\vec{t}}}(I\cap (X^{\vec{t}}))\right)^n=(I:X^{\vec{t}})^n.
	\end{equation*}
	
	If $\fm$ is associated to $I:X^{\vec{t}}$, then there exists a $X^{\vec{w}}$ such that
	\begin{align*}
	\fm = (I:X^{\vec{t}}):X^{\vec{w}}=I:X^{\vec{t}}X^{\vec{w}}=I:X^{\vec{t}+\vec{w}}
	\end{align*}
	holds and therefore $\fm\in\Ass(R/I)$.
	
	Conversely, let $\fm\in\Ass(R/I)$ with witness $X^{\vec{w}}$. Since all generators of $I$ are divisible by $X^{\vec{t}}$ and $X^{\vec{w}}X_i\in I$, we have $\vec{t}\le \vec{w}+\vec{e}_i$ for all $i\in[r]$,  where $\vec{e}_i$ is the $i$-th unit vector. Hence, $\vec{t}\le \vec{w}$ and $X^{\vec{w}}=X^{\vec{\widetilde{w}}}X^{\vec{t}}$ for some~$X^{\vec{\widetilde{w}}}$.
	Therefore,
	\begin{equation*}
	\fm = I:X^{\vec{w}} = I:X^{\vec{t}}X^{\vec{\widetilde{w}}}=(I:X^{\vec{t}}):X^{\vec{\widetilde{w}}},
	\end{equation*}
        i.e., $\fm\in \Ass(R/(I:X^{\vec{t}})$.
\end{proof}

\begin{corollary}
	Let $I$ be a monomial ideal and $X^{\vec{t}}$ be a divisor of all the generators of $I$. Then
	\begin{equation*}
	\Ass(R/I^n)\setminus\{(X_1),\dots,(X_r)\}=\Ass(R/(I:X^{\vec{t}})^n)\setminus\{(X_1),\dots,(X_r)\}.
	\end{equation*}
\end{corollary}
\begin{proof}
	Let $M\subseteq[r]$ with $\lvert M\rvert>1$. We apply Lemma~\ref{lemma:verschieben} to $\fm_{R_M}$ in the localization $R_M$ of $R$ at $\fp(M)$. Observation~\ref{observation:p(M)} then yields that $\fp (M)\in\Ass(R/I^n)$ if and only if $\fp (M)\in\Ass(R/(I:X^{\vec{t}})^n)$.
\end{proof}

\begin{definition}\label{def:d-red}
	Let $d$ denote the maximum total degree of the minimal generators of
	the monomial ideal $I$ and $X^{\vec{t}}$ their greatest common divisor.  We define the
	\textbf{reduced maximal degree} of $I$ as
	\begin{equation*}
	\dred \coloneqq d-\sum_{i=1}^rt_i.
	\end{equation*}
\end{definition}

We close this subsection with two remarks concerning the relation of
associated primes with the number of variables and the number of
minimal generators of the ideal $I$.

\begin{fact}[{\cite[Lemma~2.1]{Nasernejad-Rajaee:2019:detecting-maximal-associated-ideal}}]\label{fact:more-generators-than-variables}
	If the number of generators of a monomial ideal $I$ is smaller than the number of variables, i.e., $s<r$, then $\fm\notin\Ass(R/I^n)$ for all $n\in\N$.
\end{fact}

\begin{remark}
	The stability index of a monomial ideal in a polynomial ring with
	two variables ($r=2$) is equal to 1. This follows
	from~\cite[Theorem~2.7]{Nasernejad-Rajaee:2019:detecting-maximal-associated-ideal},
	which implies that the maximal ideal $\fm$ is either associated to
	no powers or to all powers of said ideal, depending on whether it is
	a principal ideal or not.
\end{remark}

\subsection{Conditions for the maximal ideal to be associated}\label{sec:conditions}
By Observation~\ref{observation:p(M)}, we have $\fp(M)\in\Ass(R/I)$ if and only if $\fm_{R_M}\in\Ass(R_M/I_M)$.
This subsection therefore focuses on maximal ideals. We state all results for the maximal ideal $\fm$ in $R$.

It follows immediately from the definition that the set of all witnesses of $\fm$ is $I:\fm\setminus I$.
This gives the following well-known characterizing statement.
\begin{fact}\label{fact:I:m=I}
	Let $I$ be a monomial ideal.
	Then $\fm\in\Ass(R/I)$ if and only if $I:\fm \neq I$.
\end{fact}

\begin{definition}[{cf.~\cite[Section~15.10.6]{Eisenbud:2004:commutativealgebra}}]\label{def:J}
	For a monomial ideal $I$, let
	\begin{equation*}
	\sat[I]\coloneqq I:\fm^{\infty}=\bigcup_{k\in \N_0} (I:\fm^k)
	\end{equation*}
	be the \textbf{saturation} of $I$ with respect to $\fm$.
\end{definition}
\begin{remark}\label{remark:sat(I)=intersection}
Since $\sat[I]=\bigcap_{i=1}^r (I:X_i^{\infty})$ holds ({cf.~\cite[Lemma~3.5.12]{Kreuzer-Robbiano-2008-comp-comm-algebra}}), $\sat[I]$ is again a monomial ideal.
\end{remark}
\begin{fact}[{\cite[Chapter~4, Exercise~14]{Cox-Little-OShea:2015:ideals-varieties-algorithms}}]\label{prop:equiv-m-Jn}
	Let $I$ be a monomial ideal. Then 
	\begin{equation*}
	\fm\in\Ass(R/I)\quad\Longleftrightarrow\quad \sat[I]\neq I.
	\end{equation*}
\end{fact}

\begin{lemma}\label{lemma:Jn-vs-Jn-cap-I}
	Let $I$ be a monomial ideal. For any $n\in\N$, $\sat[I^n] \neq I^n$ if and only if  $\sat[I^n]\cap I^{n-1}\neq I^n$.
\end{lemma}
\begin{proof}
	Note that the following inclusions hold:
	\begin{equation*}
	I^n\subseteq \sat[I^n]\cap I^{n-1}\subseteq \sat[I^n].
	\end{equation*}
	Therefore, $\sat[I^n]\cap I^{n-1}\neq I^n$ implies $\sat[I^n]\neq I^n$.
	For the reverse implication, note that $I^n = I^n:\fm$ implies $I^n = \sat[I^n]$. Hence, if $I^n\neq\sat[I^n]$ then
	there exists a monomial $X^{\veca}\in (I^n:\fm)\setminus I^n$. Let $i\in[r]$. Since $X^{\veca} X_i\in I^n$, there exist $u\in K$ and generators $g_1$, \dots, $g_n$ of $I$ such that
	\begin{align*}
	X^{\veca} X_i = ug_1\cdots g_n.
	\end{align*}
	This further implies that $X_i\mid g_j$ for some $j$, say $j=1$, and hence
	\begin{align*}
	X^{\veca}=u\frac{g_1}{X_i}g_2\cdots g_n\in I^{n-1}.
	\end{align*}
	Since $I^n:\fm \subseteq \sat[I^n]$, we conclude that $X^{\veca}\in (\sat[I^n]\cap I^{n-1})\setminus I^n$.
\end{proof}

\begin{remark}\label{remark:characterization-m-associated}
	We can now provide a list of statements that characterize when the maximal ideal $\fm$ is associated to $I^n$.
	The following statements are equivalent:
	\begin{enumerate}
		\item $\fm\in\Ass(R/I^n)$,
		\item\label{condition-I:m} $I^n:\fm \neq I^n$ (Fact~\ref{fact:I:m=I}),
		\item\label{condition-J:}  $\sat[I^n] \neq I^n$ (Fact~\ref{prop:equiv-m-Jn}),
		\item\label{condition-JcapI} $\sat[I^n]\cap I^{n-1}\neq I^n$ (Lemma~\ref{lemma:Jn-vs-Jn-cap-I}).
	\end{enumerate}
\end{remark}
Let $U_n$ be one of the three sets $I^n:\fm$, $\sat[I^n]$ or $\sat[I^n]\cap I^{n-1}$.
Note that $I^n\subseteq U_n$ holds in all three cases.
By Remark~\ref{remark:characterization-m-associated}, $\fm\in\Ass(R/I^n)$ if and only if the homogeneous component of degree $n$ of the graded module
\begin{equation*}
	\bigoplus_{i\ge 0}\left(U_i/I^i\right)t^i
\end{equation*}
is nonzero, where $t$ is the grading variable.
In Section~\ref{section:upper-bound} we describe this graded module via solutions of systems of linear inequalities $A\vec{x}\le \vec{b}$.
Before, in Section~\ref{sec:bounds}, we develop theory about sizes of integer solutions of such systems and translate it to the setting of graded modules.
We then use this to get insights on the behavior of $U_n/I^n$ for increasing powers $n\in\N$.

\section{Graded factor modules related to systems of linear inequalities}\label{sec:bounds}
With the overall goal to describe the ideals mentioned in Remark~\ref{remark:characterization-m-associated} via systems of linear inequalities, we start with a more generic set-up and build the connection to said ideals later (Example~\ref{example:I^n} and Section~\ref{section:upper-bound}).
The methods developed here are built on the techniques mentioned by Bruce Fields~\cite[Section~7]{Fields:2002:Lengths-of-Tors} which are further utilized by Lê~Tuân~Hoa~\cite{Hoa:2006:stab-assprimes-monom} to obtain an upper bound for the stability index (Remark~\ref{remark:bound-Hoa}).

\begin{convention}\label{convention:system}
	Throughout this section, $A\vec{x}\le\vec{b}$ denotes a system of (componentwise) inequalities, where $A\in\Z^{m\times\nu}$ and $\vec{b}\in\N_{0}^{m}$.	
\end{convention}
\begin{remark}
	We are interested in non-negative integer solutions of such a system which can always be enforced by the additional constraints $-I_{\nu}\vec{x}\le \vec 0$, where $I_{\nu}$ denotes the $\nu\times\nu$ identity matrix.
	For reasons of readability, however, we omit these additional rows and intersect the solution space with $\N_0^{\nu}$ instead.
\end{remark}

\begin{definition}\label{def:S,Sb}
	Let $A\in\Z^{m\times\nu}$. For any $\vec{b}\in\N_{0}^{m}$, we denote the set of all integer solutions of the system by
	\begin{align*}
	S_{\vec{b}}\coloneqq\{\vec{x}\in\mathbb{N}_0^{\nu}\mid A\vec{x}\le \vec{b}\}.
	\end{align*}
	Furthermore, we define the following subset of the polynomial ring $K[W_1,\dots,W_{\nu}]$:
	\begin{align*}
	\Sb\coloneqq \Span_K\!\left\{W^{\vec{x}}\mid \vec{x}\in S_{\vec{b}}\right\}.
	\end{align*}
\end{definition}

\begin{remark}
	Note that all the sets introduced in Definition~\ref{def:S,Sb} depend on the matrix $A$. However, for the sake of readability, we omit this dependence from the notation.
\end{remark}

\begin{remark}\label{remark:S,Sb}
	We observe that the set $S_{\vec{0}}=\{\vec{x}\in\mathbb{N}_0^{\nu}\mid A\vec{x}\le \vec{0}\}$ is a submonoid of $\N_0^{\nu}$, because $\vec{0}\in S_{\vec{0}}$ and
	for $\vec{x}_1$, $\vec{x}_2\in S_{\vec{0}}$ we have $A(\vec{x}_1+\vec{x}_2)=A\vec{x}_1+A\vec{x}_2\le \vec{0}$, hence
	\begin{equation}\label{eq:S}
	\vec{x}_1+\vec{x}_2\in S_{\vec{0}}.
	\end{equation}
	This implies that $\SO$ is a ring.

	Note that $S_{\vec{0}}\subseteq S_{\vec{b}}$ since $\vec{b}\ge \vec{0}$.
	Furthermore, if $\vec{x}\in S_{\vec{0}}$ and $\vec{y}\in S_{\vec{b}}$, then it follows that $A(\vec{x}+\vec{y})=A\vec{x}+A\vec{y}\le \vec{0}+\vec{b}=\vec{b}$. Hence,
	\begin{equation}\label{eq:S_b}
	\vec{x}+\vec{y}\in S_{\vec{b}}
	\end{equation}
	holds which in turn implies that $\Sb$ is an $\SO$-module.
\end{remark}

\begin{example}\label{example:I^n}
	Let $I=(X^{\veca_1},\dots,X^{\veca_s})$ be a monomial ideal.
	Then 
	\begin{align*}
	I^n&=((X^{\veca_1})^{k_1}\cdots(X^{\veca_{s}})^{k_{s}}\mid k_1,\dots,k_{s}\in\N_0,\,n= k_1+\dots+k_{s}).
	\end{align*}
	We want to set up a system of linear inequalities that describes when a monomial $X^{\vec{h}}$ is an element of $I^n$.
	We have $X^{\vec{h}}\in I^n$ if and only if there are $k_1$, \dots, $k_{s}\in\N_0$ such that $n= k_1+\dots+k_{s}$ and $X^{\vec{h}}$ is divisible by $(X^{\veca_1})^{k_1}\cdots(X^{\veca_{s}})^{k_s}$.
	It suffices to demand that $n\le k_1+\dots+k_s$, since $I^m\subseteq I^n$ for all $m\ge n$.
	So we are looking for a non-negative integer solution of
	\begin{align*}
	a_{1,j}k_1+\dots+a_{s,j}k_s-h_j&\le0,\\
	-(k_1+\dots+k_s)+n &\le 0
	\end{align*}
	for all $j\in[r]$.
	In other words, $X^{\vec{h}}\in I^n$ if and only if there exists $\vec{k}\in\N_0^{s}$ such that
	\begin{center}
		\begin{tikzpicture}[mymatrixenv]
  \matrix[mymatrix] (m)  {
    {}         & {}        &   &      & {}        &{}  &                   &  &{0} \\
               &           &   &      &           &    &                   &  &    \\
    {\vec{a}_1}&{\vec{a}_2}&   &      &{\vec{a}_s}&    &|[scale=1.5]|{-I_r}&  &    \\
               &           &   &      &           &    &                   &  &    \\
    {}         &           &   &      &           &    &                   &  &    \\
    {}         & {}        &   &      & {}        &{}  &                   &{}&{0}\\                
    {-1}       & -1        &   &      & -1        &{0} &                   &0 &{1} \\
  };    
  \matrix[mymatrix, right =0.6cm of m] (m1)  {
    {}\\
    {\vec{k}}  \\
    {}\\[0.4cm]
    {\vec{h}} \\
    {}\\[0.2cm]
    {n}       \\
  };    

  \node[right =0.2 of m1] {$\le \vec{0}.$};

   % % dots
  \draw[dotted,shorten <=0.2cm,shorten >=0.2cm] (m-1-9.south) -- (m-6-9.north);
  \draw[dotted,shorten <=0.2cm,shorten >=0.2cm] (m-7-6.east) -- (m-7-8.west);
  \draw[dotted,shorten <=0.3cm,shorten >=0.3cm] (m-3-2.east) -- (m-3-5.west);
  \draw[dotted,shorten <=0.3cm,shorten >=0.3cm] (m-7-2.east) -- (m-7-5.west);

   % % boxes
  \draw[ForestGreen, rounded corners] ($(m1-2-1.north)+(-0.8em,1em)$) rectangle ($(m1-2-1.south)+(0.8em,-1em)$);
  \draw[ForestGreen, rounded corners] ($(m-1-1.north)+(-1em,0.9em)$) rectangle ($(m-7-5)+(0.9em,-0.5em)$);

  \draw[RedViolet, rounded corners] ($(m-1-6.north)+(-0.5em,0.9em)$) rectangle ($(m-7-8.south)+(0.4em,-0.18em)$);
  \draw[RedViolet, rounded corners] ($(m1-4-1.north)+(-0.8em,1em)$) rectangle ($(m1-4-1.south)+(0.8em,-1em)$);

  \draw[RoyalPurple, rounded corners] ($(m-1-9.north)+(-0.7em,0.3em)$) rectangle ($(m-7-9.south)+(0.7em,-0.18em)$);
  \draw[RoyalPurple, rounded corners] ($(m1-6-1.north)+(-0.8em,0.45em)$) rectangle ($(m1-6-1.south)+(0.8em,-0.25em)$);

\end{tikzpicture}

%%% Local Variables:
%%% mode: latex
%%% TeX-master: "bounding-the-copersistence-index.tex"
%%% End:
	\end{center}
	Given a solution to this system of linear inequalities, the key information we need—the exponents $\vec{h}$ and the power $n$—is stored in the last $r+1$ components of the solution. For the first $s$ components, only their existence matters, not their exact values.
\end{example}

\begin{definition}
	Let $r\in\N_0$ with $r<\nu$ and $\pi_r\colon\N_0^{\nu}\to\N_0^{r+1}$ be the projection of a $\nu$-tuple onto its last $r+1$ entries, i.e., $\pi_r((x_1,\dots,x_{\nu}))=(x_{\nu-r},\dots,x_{\nu})$.
\end{definition}

\begin{definition}\label{def:H,U}
	Let $A\in\Z^{m\times\nu}$. For any $\vec{b}\in\N_{0}^{m}$, we define
	\begin{align*}
	\mathcal{H}_{\vec{b}}\coloneqq \Span_K\bigl\{X^{\pi_r(\vec{z})}\mid \vec{z}\in S_{\vec{b}}\bigr\}\subseteq K[X_1,\dots, X_{r+1}].
	\end{align*}
	In particular,
	$\mathcal{H}_{\vec{0}}=\Span_K\bigl\{X^{\pi_r(\vec{z})}\mid \vec{z}\in S_{\vec{0}}\bigr\}$.
	By setting $\deg X^{\pi_r(\vec{z})}=\pi_0(\vec{z})$ we impose a grading on $\mathcal{H}_{\vec{b}}$.
	This gives
	\begin{align*}
	\mathcal{H}_{\vec{b}}=\bigoplus_{n\ge 0} H_{\vec{b},n},
	\end{align*}
        i.e.,
        \begin{align*}
	H_{\vec{b}, n}= \Span_K\bigl\{X^{\pi_r(\vec{z})}\mid \vec{z}\in S_{\vec{b}}, \pi_0(\vec{z})=n\bigr\}.
	\end{align*}
	We call $H_{\vec{b}, n}$ and $H_{\vec{0}, n}$ the \textbf{$n$-th solution spaces corresponding to} $A\vec{x}\le \vec{b}$.
\end{definition}

\begin{remark}
	For every $n\in\N_{0}$ the sets $H_{\vec{0},n}$ and $H_{\vec{b},n}$ are additive subgroups of $K[X_1,\dots,X_{r+1}]$ and the following properties hold:
	\begin{enumerate}
		\item $H_{\vec{0},n}\subseteq H_{\vec{b},n}$ since $S_{\vec{0}}\subseteq S_{\vec{b}}$ (Remark~\ref{remark:S,Sb}).
		\item For all $n$, $m\in\N_0$ we have that
		\begin{equation*}
		H_{\vec{0},m}H_{\vec{0},n}\subseteq H_{\vec{0},m+n}.
		\end{equation*}
		holds by Equation~\eqref{eq:S} in Remark~\ref{remark:S,Sb}. Therefore, $\mathcal{H}_{\vec{0}}$ is a graded subring of the graded ring $K[X_1,\dots, X_{r+1}]$ (graded in $X_{r+1}$).
		\item By Equation~\eqref{eq:S_b} in Remark~\ref{remark:S,Sb}
		\begin{equation*}
		H_{\vec{0},n}H_{\vec{b},m}\subseteq H_{\vec{b},n+m}
		\end{equation*}
		holds for all $n$, $m\in\N_0$, hence $\mathcal{H}_{\vec{b}}$ is a graded $\mathcal{H}_{\vec{0}}$-module.
	\end{enumerate}
\end{remark}

\begin{example}
	In Example~\ref{example:I^n} we have $\vec{b}=\vec{0}$ and $H_{\vec{0},n}= I^nX_{r+1}^n$ for all $n\in\N_0$. Therefore
	\begin{equation*}
	\mathcal{H}_{\vec{0}} = \bigoplus_{n\ge0}I^nX_{r+1}^{n},
	\end{equation*}
	which is known as the \textbf{Rees algebra} of $I$.
\end{example}

\subsection{Estimates on the generators of the solution spaces}
Next, we want to investigate the sizes of the generators of $\Sb$ as an $\SO$-module with respect to their maximal entry.
We use the fundamental fact from linear programming that polyhedra can be decomposed into a sum of a finitely generated convex hull and cone (denoted by $\conv$ and $\cone$, respectively).
More precisely, we use known bounds for the entries of their respective generators.
We summarize this in the following fact. For details we refer to the proof of~\cite[Theorem~17.1]{Schrijver:1999:linintprogr}.
\begin{fact}[{\cite[Proof of Theorem~17.1]{Schrijver:1999:linintprogr}\label{thm:schrijver}}]
	Let $A\in \mathbb{Z}^{m\times \nu}$, $\vec{b}\in\mathbb{Z}_{\ge 0}^m$ and $P = \{\vec{x}\in \mathbb{Q}_{\ge 0}^\nu\mid A\vec{x}\le \vec{b}\}$.
	Let $\Delta$ denote the maximum absolute value of the subdeterminants of the matrix $(A\mid \vec{b})$.
	
	Then there exist $\vec{z}_1$, \dots, $\vec{z}_{\ell}\in P$ and $\vec{y}_1$, \dots, $\vec{y}_s\in S_{\vec{0}}$ with all components at most $\Delta$ in absolute value
	such that
	\begin{equation*}
	P = \conv\{\vec{z}_1,\dots, \vec{z}_{\ell}\} + \cone\{\vec{y}_1,\dots, \vec{y}_s\}
	\end{equation*}
	holds.
	
	Moreover, every $\vec{x}\in S_{\vec{b}}$ can be written as $\vec{x} = \widetilde{\vec{x}}+\vec{y}$ with $\vec{y}\in S_{\vec{0}}$ and $\widetilde{\vec{x}}\in M\cap S_{\vec{b}}$, where
	\begin{equation}\label{eq:bounded-generators}
	M=\conv\{\vec{z}_1,\dots, \vec{z}_{\ell}\} + \left\{\sum_{i=1}^{s}\alpha_i\vec{y}_i\mid 0\le\alpha_i<1\text{, at most }\nu\text{ of the }\alpha_i\text{ are nonzero} \right\}.
	\end{equation}
\end{fact}
For our purposes,~\eqref{eq:bounded-generators} is crucial: the set $M$ is bounded because the maximum norm of each of $\vec{z}_1$, \dots, $\vec{z}_{\ell}$ as well as $\vec{y}_1$,
\dots, $\vec{y}_s$ is bounded by $\Delta$; therefore the maximum norm of all vectors
in $M$ is bounded by $\Delta(\nu+1)$. Note that while
$\vec{z}_1$, \dots, $\vec{z}_{\ell}$ might be rational vectors, the set
$M\cap S_{\vec{b}}$ consists of integer vectors by definition.
Rewriting $\vec{x} = \widetilde{\vec{x}}+\vec{y}$ in terms of $\SO$ and $\Sb$ immediately leads to the
following corollary.
\begin{corollary}[cf.~{\cite[Lemma~2.2]{Hoa:2006:stab-assprimes-monom}}]\label{thm:degS-Sb}
	Let $A\vec{x}\le \vec{b}$ be a system as in Convention~\ref{convention:system}.
	
	Then the $\SO$-module $\Sb$ is generated by finitely many monomials all of whose exponents are at most $\sigma\coloneqq \Delta(A\mid \vec{b})(\nu+1)$, where $\Delta(A\mid \vec{b})$ is the maximum absolute value of the subdeterminants of $(A\mid \vec{b})$ and $\nu$ is the number of columns of $A$.
\end{corollary}

We use Corollary~\ref{thm:degS-Sb} to get a bound for the degree of the generators of $\mathcal{H}_{\vec{b}}$ as an $\mathcal{H}_{\vec{0}}$-module.
\begin{proposition}\label{prop:generators-of-U}
	Let $A\vec{x}\le \vec{b}$ be as in Convention~\ref{convention:system}.
	Then $\mathcal{H}_{\vec{b}}$ is generated as an $\mathcal{H}_{\vec{0}}$-module by homogeneous elements whose degree is less than or equal to
	\begin{equation*}
	\sigma\coloneqq \Delta(A\mid \vec{b})(\nu + 1),
	\end{equation*}
	where $\Delta(A\mid \vec{b})$ is the maximal absolute value of the subdeterminants of the matrix $(A\mid \vec{b})$.
\end{proposition}
\begin{proof}
	To simplify notation, we write $\pi$ instead of $\pi_r$ within this proof. We restrict the ring epimorphism
	\begin{align*}
	\varphi\colon K[W_1,\dots, W_{\nu}] &\to K[X_1,\dots, X_{r+1}]\\
	W^{\vec{z}} &\mapsto X^{\pi(\vec{z})}
	\end{align*}
	to $\Sb$ resulting in an epimorphism of
	additive groups
	\begin{align*}
	\varphi'\colon \Sb&\to \mathcal{H}_{\vec{b}}\\
	W^{\vec{z}} &\mapsto X^{\pi(\vec{z})}.
	\end{align*}
	A further restriction to the ring $\SO$ results in
	the ring epimorphism
	\begin{equation*}
	\varphi''\colon \SO\to \mathcal{H}_{\vec{0}}.
	\end{equation*}
	Let $L$ be the kernel of $\varphi''$.  Then
	$\varphi(L\Sb) =
	\varphi(L)\varphi(\Sb) = \{0\}$. Therefore,
	$L\Sb$ is a subgroup of the kernel of
	$\varphi'$. In fact, $L\Sb$ is an
	$\SO$-submodule of the kernel of $\varphi'$ and
	$\Sb/L\Sb$ is an
	$\SO/L$-module.
	
	By Corollary~\ref{thm:degS-Sb}, there exist elements
	$\vec{z}_1$, \dots, $\vec{z}_{\ell}$ with $\pi_0(\vec{z}_i)\le \sigma$ for $1\le i\le\ell$ such that
	$\Sb$ is generated as an $\SO$-module
	by $W^{\vec{z}_1}$, \dots, $W^{\vec{z}_{\ell}}$.  Hence
	$\Sb/L\Sb$ is generated as
	an $\SO/L$-module by
	$W^{\vec{z}_1}+L\Sb$, \dots,
	$W^{\vec{z}_{\ell}}+L\Sb$.  By the
	isomorphism induced by $\varphi''$ we have
	\begin{align*}
	\SO/L\simeq \mathcal{H}_{\vec{0}}
	\end{align*}
	and furthermore, since $\varphi'$ is surjective, we get that
	$X^{\pi(\vec{z}_1)}$, \dots, $X^{\pi(\vec{z}_{\ell})}$
	are generators of $\mathcal{H}_{\vec{b}}$ as an $\mathcal{H}_{\vec{0}}$-module.  The
	isomorphism and the correlation between the generators is visualized
	in the commutative diagrams in Figure~\ref{figure:comm-diagrams}.
	\begin{figure}[h]
		\centering
		\begin{tikzpicture}
		\input{figure-diagram-proof-generators}
		\end{tikzpicture}
		\caption{Depiction of the argument in the proof of
			Proposition~\ref{prop:generators-of-U} how
			$\SO$-generators of
			$\Sb$ are mapped to
			$\mathcal{H}_{\vec{0}}$-generators of $\mathcal{H}_{\vec{b}}$ via
			$\SO/L$-generators of
			$\Sb/L\Sb$ .}
		\label{figure:comm-diagrams}
	\end{figure}

	This completes the proof since $\deg(X^{\pi(\vec{z}_i)}) = \pi_0(\vec{z}_i)\le \sigma$ holds for all $i\in\{1,\dots,\ell\}$.
\end{proof}

\begin{remark}
	The special case of Proposition~\ref{prop:generators-of-U} applied to $H_{\vec{0},n}=I^n$ and $H_{\vec{b},n}=\sat[I^n]\cap I^{n-1}$ is already proven by Lê Tuân Hoa~\cite[Lemmata~3.3 and~3.4, Proposition~3.1]{Hoa:2006:stab-assprimes-monom}.
\end{remark}

\subsection{Homogeneous elements of the factor module \texorpdfstring{$\mathcal{H}_{\vec{b}}/\mathcal{H}_{\vec{0}}$}{}}
The ring $\mathcal{H}_{\vec{0}}$ is an $\mathcal{H}_{\vec{0}}$-submodule of $\mathcal{H}_{\vec{b}}$.
A straight-forward verification yields
\begin{equation*}
\mathcal{H}_{\vec{b}}/\mathcal{H}_{\vec{0}}\simeq\bigoplus_{n\ge0}(H_{\vec{b},n}/H_{\vec{0},n}).
\end{equation*}
This is a graded $\mathcal{H}_{\vec{0}}$-module with scalar multiplication
\begin{align*}
h_n\cdot(u_m+H_{\vec{0},m})\coloneqq h_nu_m+H_{\vec{0},m+n}
\end{align*}
for $h_n\in H_{\vec{0},n}$ and $u_m\in H_{\vec{b},m}$.
Again, the maximal degree of the generators of $\mathcal{H}_{\vec{b}}/\mathcal{H}_{\vec{0}}$ is bounded by the value $\sigma$ given in Proposition~\ref{prop:generators-of-U}
because the generators of $\mathcal{H}_{\vec{b}}$ as an $\mathcal{H}_{\vec{0}}$-module map to generators of $\mathcal{H}_{\vec{b}}/\mathcal{H}_{\vec{0}}$ under the projection modulo $\mathcal{H}_{\vec{0}}$.

\begin{proposition}\label{prop:U_n/H_n=0}
	Let $A\vec{x}\le \vec{b}$ be as in Convention~\ref{convention:system} such that the corresponding solution spaces fulfill
	$H_{\vec{0},m}H_{\vec{0},n}=H_{\vec{0},m+n}$ for all $n$, $m\in\N_0$.
	
	Then the following property holds:
	If $H_{\vec{b},n}/H_{\vec{0},n}=0$ for some $n\ge \sigma=\Delta(A\mid \vec{b})(\nu+1)$, then $H_{\vec{b},N}/H_{\vec{0},N}=0$ for all $N\ge n$.
\end{proposition}
\begin{proof}
	It suffices to prove that $H_{\vec{b},n+1}/H_{\vec{0},n+1}=0$. Recall that the $\mathcal{H}_{\vec{0}}$-module  $\mathcal{H}_{\vec{b}}$ is generated by elements with degree at most $\sigma$ according to Proposition~\ref{prop:generators-of-U}.
	Thus, the homogeneous elements in $H_{\vec{b},n+1}/H_{\vec{0},n+1}$ are of the form $h_m(u_k+H_{\vec{0},k})=h_mu_k+H_{\vec{0},m+k}$, where $k+m=n+1$, $k\le \sigma$,
        $h_m\in H_{\vec{0},m}$, and $u_k\in H_{\vec{b},k}$.
	Since $h_m\in H_{\vec{0,m}}=H_{\vec{0},1}H_{\vec{0},m-1}$ we can write $h_m=ab$, where $a\in H_{\vec{0},1}$ and $b\in H_{\vec{0},m-1}$.
	Then
	\begin{equation*}
	h_mu_k+H_{\vec{0},m+k}=a\cdot bu_k
	+H_{\vec{0},n+1}=0,
	\end{equation*}
	since $bu_k\in H_{\vec{b},n}=H_{\vec{0},n}$ and $aH_{\vec{0},n}\subseteq H_{\vec{0},n+1}$.
\end{proof}

\section{An upper bound for the copersistence index \texorpdfstring{$\cpi{I}$}{}}\label{section:upper-bound}
In this section we combine the results of the previous two sections to obtain an upper bound for the copersistence index $\cpi{I}$, introduced in Definition~\ref{def:cpi-persindex}.
\begin{definition}\label{definition:B1(m)}
	Let $I \subseteq K[X_1, \dots, X_r]$ be a monomial ideal and $M\subseteq[r]$. Then we denote by $\cpi{I}(M)\in\N$ the smallest number such that the following statement holds:
	If for some $N\ge\cpi{I}(M)$ the prime ideal $\fp (M)$ is not associated to $I^N$, then it follows that for all $n\ge N$
	\begin{equation*}
	\fp (M) \notin\Ass(R/I^{n}).
	\end{equation*}
\end{definition}

\begin{remark}\label{remark:bound-min-rs}
	It follows from the definition that $\cpi{I} = \max\left\{ \cpi{I}(M) \mid M\subseteq [r]  \right\}$. We argue that in fact
        \begin{equation*}
        \cpi{I} = \max\left\{ \cpi{I}(M) \longmid M\subseteq[r]\text{ and }|M|\le s \right\}
        \end{equation*}
        where $s$ is the number of generators of $I$.

	Indeed, if $s<\lvert M\rvert \le r$ then $\fm_{R_M}\notin \Ass(R_M/I_M^n)$ for all $n$ by Fact~\ref{fact:more-generators-than-variables} since $R_M$ is a polynomial ring in $\lvert M\rvert$ variables (Remark~\ref{remark:localization}) and the ideal $I_M$ has at most $s$ generators. By Observation~\ref{observation:p(M)} it follows that $\fp(M)\notin \Ass(R/I^n)$ for all $n$, i.e., $\cpi{I}(M)=1$.
	This proves the claim.
\end{remark}

\begin{theorem}\label{theorem:bounds-sigma(m)}
	Let $I$ be a monomial ideal in $K[X_1,\dots, X_r]$ and $M\subseteq[r]$. Further, let
$A\in\Z^{m\times\nu}$ and $\vec{b}\in\Z_{\ge0}^m$ such that for every $n\in\N$ the associated $n$-th solution spaces $H_{\vec{0},n}$ and $H_{\vec{b},n}$ (Definition~\ref{def:H,U})
fulfill

\begin{enumerate}
\item\label{condition:m-ass-equivalence} $\fp(M)\in\Ass(R/I^n)$ if and
  only if $H_{\vec{b},n}/H_{\vec{0},n}\neq 0$, and
\item\label{condition:HnHm} for all $n_1$, $n_{2}\in\N$ we have
  $H_{\vec{0},n_1}H_{\vec{0},n_2}=H_{\vec{0},n_1+n_2}$.
\end{enumerate}
Then $\cpi{I}(M)\le \Delta(A\mid \vec{b})(\nu+1)$.

%%% Local Variables:
%%% mode: latex
%%% TeX-master: "bounding-the-copersistence-index.tex"
%%% End:
\end{theorem}
\begin{proof}
	This theorem is a special case of Proposition~\ref{prop:U_n/H_n=0}, where the system matrix $A$ and the right-hand side $\vec{b}$ are chosen such that the conditions~\eqref{condition:m-ass-equivalence} and \eqref{condition:HnHm} are satisfied.
\end{proof}

The obtained bound from Theorem~\ref{theorem:bounds-sigma(m)} depends on the system $A\vec{x}\le\vec{b}$. The next step is the following: We focus on $\fm$ rather than on all prime ideals $\fp(M)$, having Observation~\ref{observation:p(M)} in mind. 
For certain systems it is then possible to further estimate $\cpi{I}([r])$ by a new bound $\sigma(d,s,r)$ that only depends on the number of variables $r$, the number $s$ of generators of $I$ and their maximal total degree $d$. As the next lemma states, whenever such a function $\sigma$ exists that is non-decreasing in all variables, this yields a bound for the copersistence index.

\begin{proposition}\label{proposition:sigma}
	Let $\sigma\colon\N^3\to\N$ be a map that is non-decreasing in all three
variables such that for all $d$, $s$, $r\in\N$ the inequality $\cpi{I}([r])\le\sigma(d,s,r)$ holds, whenever $I$ is a monomial ideal in $r$ variables, $s$ generators, and whose minimal generators have total degree at most $d$.	
Then $\sigma(d,s,r)$ is an upper bound for the copersistence index
$\cpi{I}$ of every such ideal $I$.

%%% Local Variables:
%%% mode: latex
%%% TeX-master: "bounding-the-copersistence-index.tex"
%%% End:
\end{proposition}
\begin{proof}
	Let $M \subseteq [r]$ and $I_M$ the ideal generated by $I$ in the localization $R_M$ of $R$ at $\fp(M)$. Further, let $\tilde s$ the number of minimal generators of $I_M$ and $\tilde d$ be their maximal total degree. Then $\tilde d \le d$ and $\tilde s \le s$. Moreover, $I_M$ is an ideal in $R_M$ which is a polynomial ring in $|M|\le r$ variables. We can conclude that  
	\begin{equation*}
		\cpi{I}(M) = \cpi{I_M}(M) \le \sigma(\tilde d, \tilde s, |M|) \le \sigma(d,s,r),
	\end{equation*}
	where the leftmost equality is due to Observation~\ref{observation:p(M)} and the middle and rightmost inequality follow from the hypotheses of the proposition.
	Since $\cpi{I} = \max\!\left\{\cpi{I}(M) \mid M\subseteq [r]\right\}$, this finishes the proof.
\end{proof}

In order to find a suitable function $\sigma$ to bound $\cpi{I}([r])$, we set up suitable systems  $A\vec{x} \le \vec{b}$ for Theorem~\ref{theorem:bounds-sigma(m)} to be applicable with $M=[r]$. 
In Remark~\ref{remark:characterization-m-associated} we gave three statements that characterize $\fm\in\Ass(R/I^n)$.
They are of the form
\begin{equation*}
\fm\in\Ass(R/I^n)\quad\Longleftrightarrow\quad U_n/I^n\neq0
\end{equation*}
where $U_n \in \left\{\sat[I^n]\cap I^{n-1}, I^n : \fm, \sat[I^n] \right\}$. The resulting bounds are discussed below in Subsections~\ref{subsec:Un=JnI^n-1},~\ref{subsec:Un=I:m}, and~\ref{subsec:Un=Jn}. We point out that even for a fixed choice of $U_n$ there are in general multiple options to set up a suitable system $A\vec{x} \le \vec{b}$. We restrict our investigation to specific choices.

\begin{notation}\label{notation:X_r+1=t}
In contrast to Section~\ref{sec:bounds}, we write $X=(X_1,\dots,X_r)$ and $t$ instead of $X_{r+1}$ to distinguish notationally between the variables of the ambient ring $K[X_1,\dots,X_r]$ of the ideal~$I$ and the variable $t$ we use for the grading of $\mathcal{H}_{\vec{b}}$. 	
\end{notation}

\subsection{Approach 1\texorpdfstring{: $\fm\in\Ass(R/I^n)$ if and only if $(\sat[I^n]\cap I^{n-1})\neq I^n$}{}}\label{subsec:Un=JnI^n-1}
As mentioned in Remark~\ref{remark:bound-Hoa}, this is the approach followed by Lê~Tuân~Hoa~\cite{Hoa:2006:stab-assprimes-monom}. He sets up a system $A\vec{x}\le\vec{b}$ satisfying Convention~\ref{convention:system} and then determines a bound for $\Delta(A \mid \vec{b})$. For more details, we refer to the original proofs.
His argument yields the upper bound for the copersistence index $\cpi{I}$
	\begin{equation*}
	\sigma_1(d,s,r)\coloneqq d(rs+s+d)(\sqrt{r})^{r+1}(\sqrt{2}d)^{(r+1)(s-1)},
	\end{equation*}
        where $r$ is the number of variables, $s$ is the number of generators of $I$ and $d$ is their maximal total degree.
Lemma~\ref{lemma:verschieben} and Remark~\ref{remark:bound-min-rs} further imply
	\begin{equation*}
		\cpi{I}\le \sigma_1(\dred,s,\min\{r,s\}).
	\end{equation*}

\subsection{Approach 2\texorpdfstring{: $\fm\in\Ass(R/I^n)$ if and only if $(I^n:\fm)\neq I^n$}{}}\label{subsec:Un=I:m}
We set up a system of linear constraints $A\vec{x}\le \vec{b}$ such that $H_{\vec{b},n}=(I^n:\fm) t^n$ and later show that $H_{\vec{0},n}=I^nt^n$ (in the notation of Section~\ref{sec:bounds} with Notation~\ref{notation:X_r+1=t}).
A first idea how to set up such a system was introduced in Example~\ref{example:I^n}.

If $I=(X^{\veca_1},\dots, X^{\veca_s})$, then $I^n=(X^{k_1\veca_1+\dots+k_s\veca_s}\mid k_i\in\N_0,\, k_1+\dots+k_s=n)$.
A monomial $X^{\vec{h}}$ is an element of $I^n:\fm$ if and only if $X^{\vec{h}}X_i\in I^n$ for all $i\in[r]$, i.e., there exists a generator of $I^n$ that divides $X^{\vec{h}}X_i$.
That is, for every $i\in[r]$ there exist $k_{i,1},$ \dots, $k_{i,s}\in\N_0$ such that
$k_{i,1}+\dots+k_{i,s}=n$ and
\begin{align*}
X^{k_{i1}\veca_1+\dots+k_{is}\veca_s}\mid X^{\vec{h}+\vec{e}_i},
\end{align*}
where $\vec{e}_i\in\Z^r$ is the $i$-th unit vector.
This is equivalent to the componentwise inequality
\begin{align*}
k_{i,1}\veca_1+\dots+k_{i,s}\veca_s\le \vec{h}+\vec{e}_i.
\end{align*}
It suffices to demand that $k_{i,1}+\dots+k_{i,s}\ge n$, since $I^m\subseteq I^n$ holds for all $m\ge n$.

In conclusion:
A monomial $X^{\vec{h}}$ is an element of $I^n:\fm$ if and only if for every $i\in[r]$ there exist  $k_{i,1},\dots, k_{i,s}\in\N_0$ such that
\begin{align*}
k_{i,1}\veca_1+\dots+k_{i,s}\veca_s&\le \vec{h}+\vec{e}_i,\\
n-(k_{i,1}+\dots+k_{i,s})&\le 0.
\end{align*}
So for every $i\in[r]$ we get a block of inequalities:

\begin{center}
  \begin{tikzpicture}[mymatrixenv]
  \matrix[mymatrix] (m)  {
    {}         & {}        &   &      & {}        &{}  &                   &  &{0} \\
               &           &   &      &           &    &                   &  &    \\
    {\vec{a}_1}&{\vec{a}_2}&   &      &{\vec{a}_s}&    &|[scale=1.5]|{-I_r}&  &    \\
               &           &   &      &           &    &                   &  &    \\
    {}         &           &   &      &           &    &                   &  &    \\
    {}         & {}        &   &      & {}        &{}  &                   &{}&{0}\\                
    {-1}       & -1        &   &      & -1        &{0} &                   &0 &{1} \\
  };    
  \matrix[mymatrix, right =0.6cm of m] (m1)  {
    {k_{i,1}}  \\
    \vdots\\
    {k_{i,s}}  \\
    {h_{1}}  \\
    \vdots\\
    {h_{r}}  \\
    {n}       \\
  };    

  \node[right =0.2 of m1] {$\le$};
  \matrix[mymatrix, right =1cm of m1] (m2)  {
    \vec{e}_i  \\
    0  \\
  };

   % % dots
  \draw[dotted,shorten <=0.2cm,shorten >=0.2cm] (m-1-9.south) -- (m-6-9.north);
  \draw[dotted,shorten <=0.2cm,shorten >=0.2cm] (m-7-6.east) -- (m-7-8.west);
  \draw[dotted,shorten <=0.3cm,shorten >=0.3cm] (m-3-2.east) -- (m-3-5.west);
  \draw[dotted,shorten <=0.3cm,shorten >=0.3cm] (m-7-2.east) -- (m-7-5.west);

   % % boxes
  \draw[ForestGreen, rounded corners] ($(m1-1-1.north)+(-1em,0.2em)$) rectangle ($(m1-3-1.south)+(1em,-0.1em)$);
  \draw[ForestGreen, rounded corners] ($(m-1-1.north)+(-1em,0.9em)$) rectangle ($(m-7-5)+(0.9em,-0.5em)$);

  \draw[RedViolet, rounded corners] ($(m-1-6.north)+(-0.5em,0.9em)$) rectangle ($(m-7-8.south)+(0.4em,-0.18em)$);
  \draw[RedViolet, rounded corners] ($(m1-4-1.north)+(-1em,0.1em)$) rectangle ($(m1-6-1.south)+(1em,-0.05em)$);

  \draw[RoyalPurple, rounded corners] ($(m-1-9.north)+(-0.7em,0.3em)$) rectangle ($(m-7-9.south)+(0.7em,-0.18em)$);
  \draw[RoyalPurple, rounded corners] ($(m1-7-1.north)+(-1em,0.25em)$) rectangle ($(m1-7-1.south)+(1em,-0.3em)$);

\end{tikzpicture}

%%% Local Variables:
%%% mode: latex
%%% TeX-master: "bounding-the-copersistence-index.tex"
%%% End:
\end{center}

We now combine these blocks to obtain a matrix representation of the system of inequalities.
\begin{notation}\label{notation:Ab}
  Let $m= (r+1)r$ and $\nu=rs+r+1$.  We define the matrix
  $\indi{A} \in\Z^{m+\nu\times \nu}$ and the vector $\indi{\vec{b}} \in\Z^{m+\nu}$ in the following way:

  \begin{center}
    \begin{tikzpicture}[mymatrixenv]
  \matrix[mymatrix] (m)  {
    |[alias=b1lo]|{}   &|[alias=vb1o]|{}&|[alias=b1ro]|{}   & {}                & & {}                &                               &  {}               &&  {}               &|[alias=ro1]| {}     & |[alias=ror1]|{}\\
                       &                &                   &                   & &                   &                               &                   &&                   &                     &   \\
    \veca_1            &                & \veca_s           &                   & &                   &                               &                   &&                   &|[scale=1.5]|{-I_r}  &   \\
                       &                &                   &                   & &                   &                               &                   &&                   &                     &   \\[0.2cm]
    |[alias=m1ul1]|{-1}&                &|[alias=m1ur1]|{-1}&                   & &                   &                               &                   &&                   & 0 \hspace*{0.8cm}0  &|[alias=r1]|{1}\\
    {}                 &                &                   &|[alias=b2lo]|{}   & &|[alias=b2ro]|{}   &                               &                   &&                   &|[alias=ro2]|        & |[alias=ror2]|{}\\
                       &                &                   &                   & &                   &                               &                   &&                   &                     &   \\
                       &                &                   &\veca_1            & &  \veca_s          &                               &                   &&                   &|[scale=1.5]|{-I_r}  &   \\
                       &                &                   &                   & &                   &                               &                   &&                   &                     &   \\[0.2cm]
    {}                 &                &                   &|[alias=m1ul2]|{-1}& &|[alias=m1ur2]|{-1}&                               &                   &&                   &0 \hspace*{0.8cm}0   &|[alias=r2]|{1}\\
                       &                &                   &                   & &                   &|[scale=1.5]|{\phantom{\ddots}}&                   &&                   \\                    
     {}                &                &                   &                   & &                   &                               &|[alias=b3lo]|{}   &&|[alias=b3ro]|{}   &|[alias=ro3]|        & |[alias=ror3]|{}\\
                       &                &                   &                   & &                   &                               &                   &&                   &                     &   \\
                       &|[scale=1.5]|{ }&                   &                   & &                   &                               & \veca_1           && \veca_s           &|[scale=1.5]|{-I_r}  &   \\
                       &                &                   &                   & &                   &                               &                   &&                   &                     &   \\[0.2cm]
     {}                &                &                   &                   & &                   &                               &|[alias=m1ul3]|{-1}&&|[alias=m1ur3]|{-1}& 0 \hspace*{0.8cm}0  &|[alias=r3]|{1}\\
  };

  \node at (-5.9,0) {$\indi{A} \coloneqq$};
  % % labels right
   \mymatrixbraceright{1}{5}{\footnotesize $r+1$}
   \mymatrixbraceright{6}{10}{\footnotesize $r+1$}
   \mymatrixbraceright{12}{16}{\footnotesize $r+1$}

   % % labels top
   \mymatrixbracetop{1}{3}{\footnotesize $s$}
   \mymatrixbracetop{4}{6}{\footnotesize $s$}
   \mymatrixbracetop{8}{10}{\footnotesize $s$}
   \mymatrixbracetopalt{11}{12}{\footnotesize $r+1$}
 
    % % dots
    \draw[dotted,shorten <=0.2cm,shorten >=0.2cm] (m1ul1.east) -- (m1ur1.west);
    \draw[dotted,shorten <=0.2cm,shorten >=0.2cm] (m1ul2.east) -- (m1ur2.west);
    \draw[dotted,shorten <=0.2cm,shorten >=0.2cm] (m-3-1.east) -- (m-3-3.west);
    \draw[dotted,shorten <=0.2cm,shorten >=0.2cm] (m-3-1.east) -- (m-3-3.west);
    \draw[dotted,shorten <=0.2cm,shorten >=0.2cm] (m-8-4.east) -- (m-8-6.west);
    \draw[dotted,shorten <=0.2cm,shorten >=0.2cm] (m-14-8.east) -- (m-14-10.west);
    \draw[dotted,shorten <=0.2cm,shorten >=0.2cm] (m-16-8.east) -- (m-16-10.west);
    \draw[dotted,shorten <=0.4cm,shorten >=0.4cm] (m-5-11.west) -- (m-5-11.east);
    \draw[dotted,shorten <=0.4cm,shorten >=0.4cm] (m-10-11.west) -- (m-10-11.east);
    \draw[dotted,shorten <=0.4cm,shorten >=0.4cm] (m-16-11.west) -- (m-16-11.east);

    % % boxed ddots
    \node[draw, rectangle, minimum width=0.4em, minimum height=0.4em, ForestGreen, rounded corners=1pt, inner sep=0] at ($(m1ur2.south east)+(1em,-1.1em)$) {};
    \node[draw, rectangle, minimum width=0.4em, minimum height=0.4em, ForestGreen, rounded corners=1pt, inner sep=0] at ($(m1ur2.south east)+(1.8em,-1.75em)$){} ;
    \node[draw, rectangle, minimum width=0.4em, minimum height=0.4em, ForestGreen, rounded corners=1pt, inner sep=0] at ($(m1ur2.south east)+(2.6em,-2.4em)$){} ;

    % % boxed vdots
    \node[draw, rectangle, minimum width=0.4em, minimum height=0.4em, RedViolet, rounded corners=1pt, inner sep=0] at ($(r2.south)+(-2.7em,-0.9em)$) {};
    \node[draw, rectangle, minimum width=0.4em, minimum height=0.4em, RedViolet, rounded corners=1pt, inner sep=0] at ($(r2.south)+(-2.7em,-1.75em)$) {};
    \node[draw, rectangle, minimum width=0.4em, minimum height=0.4em, RedViolet, rounded corners=1pt, inner sep=0] at ($(r2.south)+(-2.7em,-2.6em)$) {};

    % % left boxes
    \draw[ForestGreen, rounded corners] ($(b1lo.north)+(-1em,0em)$) rectangle ($(m1ur1)+(1em,  -0.69em)$);
    \draw[ForestGreen, rounded corners] ($(b2lo.north)+(-1em,0em)$) rectangle ($(m1ur2)+(1em,  -0.69em)$);
    \draw[ForestGreen, rounded corners] ($(b3lo.north)+(-1em,0em)$) rectangle ($(m1ur3)+(0.8em,-0.69em)$);

    % % Ir boxes
    \draw[RedViolet, rounded corners] ($(ro1.north)+(-1.9em,0em)$) rectangle ($(r1)+(-0.9em,-0.7em)$);
    \draw[RedViolet, rounded corners] ($(ro2.north)+(-1.9em,0em)$) rectangle ($(r2)+(-0.9em,-0.7em)$);
    \draw[RedViolet, rounded corners] ($(ro3.north)+(-1.9em,0em)$) rectangle ($(r3)+(-0.9em,-0.7em)$);

    % % n boxes
    \draw[RoyalPurple, rounded corners] ($(ror1.north)+(-0.6em,0em)$) rectangle ($(r1)+(0.7em,-0.7em)$);
    \draw[RoyalPurple, rounded corners] ($(ror2.north)+(-0.6em,0em)$) rectangle ($(r2)+(0.7em,-0.7em)$);
    \draw[RoyalPurple, rounded corners] ($(ror3.north)+(-0.6em,0em)$) rectangle ($(r3)+(0.7em,-0.7em)$);

\end{tikzpicture}

%%% Local Variables:
%%% mode: latex
%%% TeX-master: "bounding-the-copersistence-index.tex"
%%% End:
  \end{center}

  \noindent where $I_r$ denotes the $r\times r$ identity matrix. Further, we set
  \begin{equation*}
    \indi{\vec{b}}\coloneqq(\vec{e}_1^{\transpose},0,\dots, \vec{e}_r^{\transpose} ,0)^{\transpose}\in\Z^{(r+1)r}.
  \end{equation*}
\end{notation}

\begin{theorem}\label{theorem:bound-B1}
	Let $I$ be a monomial ideal in the
ring $K[X_1,\dots, X_r]$ with $s$ generators, reduced maximal degree $\dred$
(Definition~\ref{def:d-red}), and
\begin{equation*}
  \sigma_2(d,s,r)\coloneqq(\sqrt{d^2+1})^{r s}\left(\sqrt{r}\right)^{r+2}(r s+r+2). 
\end{equation*}
	
Then $\cpi{I} \le \sigma_2(\dred,s,\min\{r,s\})$ holds.

%%% Local Variables:
%%% mode: latex
%%% TeX-master: "bounding-the-copersistence-index.tex"
%%% End:
\end{theorem}
\begin{proof}
	Let $\indi{A}$ and $\indi{\vec{b}}$ be as introduced in Notation~\ref{notation:Ab}. To simplify notation, we write $A=\indi{A}$ and $\vec{b}=\indi{\vec{b}}$ in this proof.
	As explained above, $H_{\vec{b},n}=(I^n:\fm)t^n$ holds.
	
	Considering the homogeneous system
	$A\, \vec{x}\le \vec{0}$, the conditions $k_{i,1}\veca_1+\dots+k_{i,s}\veca_s\le \vec{h}+\vec{e}_i$ change to $k_{i,1}\veca_1+\dots+k_{i,s}\veca_s\le \vec{h}$, that is, $X^{k_{i,1}\veca_1+\dots+k_{i,s}\veca_s}\mid X^{\vec{h}}$.
	Hence the homogeneous system describes the set of monomials in $I^n$, that is, $H_{\vec{0},n}=I^nt^n$.

	This system satisfies the hypotheses of Theorem~\ref{theorem:bounds-sigma(m)}.
    Note that $H_{\vec{0},n} H_{\vec{0},m}= H_{\vec{0},n+m}$ holds trivially for all non-negative integers $n$ and $m$.
	Therefore,
	\begin{equation*}
		\cpi{I}([r])\le\Delta(A\mid\vec{b})(rs+r+2)
	\end{equation*}
	holds. 
	
	We use Hadamard's inequality to give an upper bound for $\Delta(A\mid\vec{b})$.
	The norms of the first $rs$ columns of $A$ are at most
	\begin{align*}
	\max_{i\in[s]}\sqrt{a_{i,1}^2+\dots+a_{i,r}^2+1}\le \sqrt{d^2+1}.
	\end{align*}
	The remaining $r+1$ columns of $A$ and $\vec{b}$ have norm $\sqrt{r}$.
	Therefore, 
	\begin{equation*}
		\cpi{I}([r])\le\Delta(A\mid \vec{b})(rs+r+2)\le (\sqrt{d^2+1})^{rs}(\sqrt{r})^{r+2}(rs+r+2)=\sigma_2(d,s,r).
	\end{equation*}
	A straight forward verification shows that $\sigma_2$ is non-decreasing in all three parameters. Hence we can apply Proposition~\ref{proposition:sigma} and obtain $\cpi{I}\le\sigma_2(d,s,r)$.
	The assertion follows from Lemma~\ref{lemma:verschieben} and Remark~\ref{remark:bound-min-rs}.
\end{proof}

\subsection{Approach 3\texorpdfstring{: $\fm\in\Ass(R/I^n)$ if and only if $\sat[I^n]\neq I^n$}{}}\label{subsec:Un=Jn}
Not too surprisingly, this approach turns out to be very similar to the one we presented in Section~\ref{subsec:Un=I:m}.
Indeed, the augmented system matrices are almost identical. The resulting bound for $\cpi{I}$ is greater than $\sigma_{2}(\dred,s,\min\{r,s\})$ of the previous subsection.
However, we briefly describe this approach here to demonstrate that there are several options to construct a system of linear inequalities that is suitable for Theorem~\ref{theorem:bounds-sigma(m)}.

By definition, $\sat[I^n]=\bigcup_{k\in\N_0}I^n:\fm^k$.
As an increasing sequence of ideals in the Noetherian ring $R$, 
the sequence $I^n:\fm^0\subseteq I^n:\fm^1\subseteq I^n:\fm^2\subseteq\cdots$ becomes stationary at some power $N\in\N$ of $\fm$.
Hence, $\sat[I^n]=\bigcup_{k=0}^N I^n:\fm^k$.
We will see in Remark~\ref{remark:A-1-in-A-2} below that the precise value of $N$ is not relevant in what follows.
By Remark~\ref{remark:sat(I)=intersection} we have $\sat[I^n]=\bigcap_{i=1}^r(I^n:X_i^{\infty})$ which implies
\begin{equation*}
	\sat[I^n]=\bigcup_{(k_1,\dots,k_r)\in\N_0^r} I^n:(X_1^{k_1},\dots,X_r^{k_r}) \supseteq \bigcup_{(k_1,\dots,k_r)\in\N_0^r,\, k_i\le N} I^n:(X_1^{k_1},\dots,X_r^{k_r}).
\end{equation*}
The reverse inclusion also holds, since if $w\in\sat[I^n]$, then there exists a $k\le N$ such that $w\in I^n:\fm^k$.
As $(X_1^k,X_2^k,\dots, X_r^k)\subseteq \fm^k$, this implies that $w\in I^n:(X_1^k,X_2^k,\dots, X_r^k)$.
We conclude
\begin{align*}
	X^{\vec{h}}\in \sat[I^n] &\Longleftrightarrow\text{ there exist }k_1,\dots,k_r\le N\text{ such that for all }i\in[r], \text{ we have }X^{\vec{h}}X_i^{k_i}\in I^n\\
	&\Longleftrightarrow X^{\vec{h}}X_i^{N}\in I^n \text{ for all }i\in[r].
\end{align*}

This is equivalent to the existence of $k_{i,1},$ \dots, $k_{i,s}\ge 0$  for all $i\in[r]$ with
\begin{align*}
-(k_{i,1}+\dots+k_{i,s})+n &\le 0,\\
k_{i,1}\veca_1+\dots+k_{i,s}\veca_s&\le \vec{h}+N\vec{e}_i.
\end{align*}

\begin{notation}	
  Let $A_{I^n:\fm}$ and $\vec{b}_{I^n:\fm}$ be as in Notation~\ref{notation:Ab}. Then we set
  \begin{equation*}
  	A_{\sat[I^n]}\coloneqq A_{I^n:\fm}\text{ and }\vec{b}_{\sat[I^n]}\coloneqq N\cdot\vec{b}_{I^n:\fm}.
  \end{equation*}
\end{notation}

\begin{remark}\label{remark:A-1-in-A-2}
	By construction, $H_{\vec{b},n}=\sat[I^n]t^n$ holds for the system $\indii{A}\,\vec{x}\le \indii{\vec{b}}$. The same argument as in Subsection~\ref{subsec:Un=I:m} yields $H_{\vec{0},n}=I^nt ^n$.
\end{remark}	

\begin{remark}\label{remark:sigma_2smaller}
	Since $\Delta(\indii{A}\mid \indii{\vec{b}})=\Delta(\indi{A}\mid N\cdot\indi{\vec{b}})\ge \Delta(\indi{A}\mid \indi{\vec{b}})$ holds, using the system $\indii{A}\,\vec{x}\le \indii{\vec{b}}$ in Theorem~\ref{theorem:bounds-sigma(m)} and the same technique as in Theorem~\ref{theorem:bound-B1} does not improve the upper bound for $\cpi{I}$ obtained in Theorem~\ref{theorem:bound-B1}.
\end{remark}
\begin{remark} 
	As pointed out earlier, there may be more than one choice to set up a system matrix. 
	Another idea was to use that $\sat[I^n]=(I^n:X_1^{\infty})\cap\dots\cap (I^n:X_r^{\infty})$. However, the corresponding system is already homogeneous and hence Theorem~\ref{theorem:bounds-sigma(m)} is not applicable.
\end{remark}

\section{Comparison of the different approaches}\label{section:compare-bounds}
 We already established that our approach in Subsection~\ref{subsec:Un=Jn} does not result in a better bound than $\sigma_2(\dred, s, \min\{r,s\})$. It remains to compare the bounds from Subsections~\ref{subsec:Un=JnI^n-1} and~\ref{subsec:Un=I:m}.

\begin{proposition}\label{thm:compare-bounds}
  Let $2\le r \le s$ and $d\ge 2$. Then
\begin{equation*}
	\sigma_2(d,s,r)<\frac{q(d)^{rs}}{\sqrt{2r}}\cdot\sigma_{2}(d,s,r)\le\sigma_{1}(d,s,r)
\end{equation*}
holds, where $q(d)\coloneqq\frac{d\sqrt{2}}{\sqrt{d^2+1}}>1$.

\end{proposition}

\begin{proof}
	Due to the hypothesis, we can estimate
	
	\begin{align*}
	\frac{\sigma_{1}(d,s,r)}{\sigma_{2}(d,s,r)}
	&=\frac{d(rs+s+d)\left(\sqrt{r}\right)^{r+1}\left(\sqrt{2}d\right)^{(r+1)(s-1)}}
	{(\sqrt{d^2+1})^{rs}\left(\sqrt{r}\right)^{r+2}(rs+\underbrace{r}_{\le s}+\underbrace{2}_{\le d})}
	\ge \frac{d^{rs+s-r}\sqrt{2}^{rs+s-r-1}}{(\sqrt{d^2+1})^{rs}\sqrt{r}}\\
	&= q(d)^{rs}\cdot \frac{d^{s-r}\sqrt{2}^{s-r}}{\sqrt{2r}}\ge \frac{q(d)^{rs}}{\sqrt{2r}}.
	\end{align*}
	This proves the second inequality in the statement of the proposition.
	For the first inequality, we show that 
	\begin{equation*}
		\frac{q(d)^{rs}}{\sqrt{2r}}>1
	\end{equation*}
	holds.
	Since $q(d)$ is increasing in $d$ and $s\ge r$, it follows that
	\begin{align*}
		\frac{q(d)^{rs}}{\sqrt{2r}}\ge \frac{q(2)^{r^2}}{\sqrt{2r}}\eqqcolon \varphi(r).
	\end{align*}
	The latter expression $\varphi(r)$ is increasing in $r$, because
	\begin{align*}
		\frac{q(2)^{r^2}}{\sqrt{2r}} < \frac{q(2)^{r^2+2r+1}}{\sqrt{2(r+1)}}\quad\Longleftrightarrow\quad r+1 < q(2)^{2(2r+1)}\cdot r,
	\end{align*}
	where the inequality on the right-hand side holds since $q(2)^{2(2r+1)}\ge q(2)^{10}=(8/5)^5>10$.
	Evaluating $\varphi(2)$ gives $64/50>1$. 
\end{proof}

\bibliographystyle{plain}
\bibliography{bibliography}
\end{document}